\documentclass[12pt,a4paper]{article}
\usepackage[english]{babel}
\usepackage{amsthm}
\usepackage{amsmath}
\usepackage{amssymb}
\usepackage{tikz}
\usetikzlibrary{shapes,arrows,positioning}
\usepackage[utf8]{inputenc}
\usepackage{tikz,pgfplots}
\usepackage{tikz-cd}
\usepackage{calc}
\usetikzlibrary{decorations.pathmorphing}
\usepackage{amsmath,systeme}
\usepackage{mathtools}
\usepackage{amssymb}
\usepackage{parskip}
\usepackage{biblatex}
\usepackage{csquotes}

\usepackage[T1]{fontenc}
\usepackage{ae}
\usepackage{aecompl}
\newtheorem{theorem}{Theorem}

\newtheorem{lemma}[theorem]{Lemma}
\newtheorem{remark}{Remark}

\pgfplotsset{compat=1.17}
\begin{document}

\title{Fermat's Last Theorem for Special Case}
\author{Alireza Sharifi}
\date{November 2020}
\maketitle

\begin{abstract}

 In this paper we present an elementary proof for a special case of Fermat's last theorem for  specific category of a, b and c. In fact, we assume that $n$ is prime and $4\rvert (n+1),$ then for $a,b$ and $c$ that $ n\nmid abc$ the equation
 $ a^{2n} = b^{2n} + c^{2n} $ does not have any solution in natural numbers. 
    
\end{abstract}

\section*{ Introduction }

 Fermat's last theorem states that for the equation $a^n = b^n + c^n$ there is no  integers a,b and c that satisfies the equation for $n>2$.  There have been significant efforts to find an elementary proof regarding this theorem and it has been done for many special cases like the case n=4 by infinite descent method or the case n=5 proved by Legendre and Dirichlet and similarly for the other special cases, some proofs were come out one after another. These strives kept going until the famous proof appeared by Andrew Wiles for general case in which elliptical curves and modern mathematics were used.

Although the version , I presented in the abstract, is special case of one specific category of a,b and c,  the proof is completely elementary and understandable for all math lovers even for undergraduate level.

A very quick tour in the proof is that, we rewrite the equation by Pythagorean triple then we expand the expressions by Binomial expansion, and we keep simplifying and using lemmas to analyze the terms till we get to the contradiction.

\[\]
\[\]
\[\]
\[\]
\[\]
\[\]
\[\]
\[\]
We start the proof by stating some lemmas. First, the following lemma is famous \emph{``Pythagorean triple"} which we assume that the reader is already convinced as though can be found in any number theory literature. 

As a matter of simplicity we use the notation $(a,b)$ in place of $gcd(a,b).$

Moreover, we agree to use the notation $(a,b,c)=1$ to say that $(a,b)=1$ , $(a,c)=1$ and $(b,c)=1$ that is to say a,b and c are pairwise relatively prime with respect to each other.

\[\]

\begin{lemma}

For any $a, b$ and $c$ $\in \mathbb{N}$ which are pairwise relatively prime,i.e. $(a,b,c)=1$ and $a^2=b^2+c^2$ there are $l$ and $t$ in $\mathbb{N} $ we can rewrite a,b and c as follows:
\[ a=l^2 + t^2 ,\qquad c=l^2-t^2 ,\qquad  b=2lt\]

\end{lemma}

We proceed by the following lemma;

\begin{lemma}

For a,b $ \in  \mathbb{N}$ where $(a,b)=1$ and for each $c,n \in \mathbb{N}$ suppose $ab= c^n$ then $\sqrt[n]{a}$ and $\sqrt[n]{b}$ are both  natural numbers.

\end{lemma}

\begin{proof}

We use prime factorization of a and b;

So, there exist prime numbers $p_1 \dots p_m $ and $ q_1 \dots q_k $, such that 

\[
     a= p_1 ^ {\alpha_1} \dots p_m ^{\alpha_m} \qquad and \qquad  b= q_1^{\beta_1} \dots q_k ^{\beta_k}
\]

by which we can write \[ ab= p_1 ^ {\alpha_1} \dots p_m ^{\alpha_m} q_1^{\beta_1} \dots q_k ^{\beta_k}  \]

clearly since $p_i$ are prime, for each $i$, $p_i\mid c$. Now let's suppose $l$ be the greatest power of $p_i$ in prime factorization of c. Then there exist $t$ such that $c= p_i^lt.$

then  \[ ab= p_1 ^ {\alpha_1} \dots p_m ^{\alpha_m} q_1^{\beta_1} \dots q_k ^{\beta_k} = p_i^{ln} t^n  \] since the prime factorization is unique, then $\alpha_i = ln.$

Now it is easy to look at the value of $\sqrt[n]{p_i^{\alpha_i}}.$

\[ \sqrt[n]{p_i^{\alpha_i}}= \sqrt[n]{p_i^{ln}}= p_i^l \in \mathbb{N} \] therefore $\sqrt[n]{a}$  $\in \mathbb{N}.$

Completely analogous argument holds for $\sqrt[n]{b}$ which we rather leave for the reader.

\end{proof}

\begin{lemma}

For each a, b in  $\mathbb{N}$ where $(a,b)=1$ and for each prime $n$ that $n\nmid (a^n-b^n) $  we have :

\[ (a-b, \frac{a^n-b^n}{a-b})=1  \]

\end{lemma}

\begin{proof}

By an elementary identity for expansion we get

\[ \frac{a^n-b^n}{a-b} = a^{n-1}+a^{n-2}b+ \dots + ab^{n-2}+ b^{n-1} \] after adding and subtracting some terms we try to dig out the term $(a-b)$ among all except the last term.

\[\Longrightarrow \qquad =a^{n-1}-a^{n-2}b+ 2a^{n-2}b-2a^{n-3}b^2+ 3a^{n-3}b^2-  \dots - (n-3)a^2b^{n-3}+\]

\[+(n-2)a^2b^{n-3}- (n-2)ab^{n-2}+ (n-1)ab^{n-2}-(n-1)b^{n-1} + nb^{n-1}.\] 

and then by factoring the term $(a-b)$ we obtain
 
  \[= a^{n-2}(a-b) + 2a^{n-3}b(a-b) + 3a^{n-4}b^2 (a-b) + \dots\] 
  \[+ (n-2)ab^{n-3}(a-b) + (n-1)b^{n-2}(a-b)+ nb^{n-1}\]
  
  Now let's suppose \[(a-b, \frac{a^n-b^n}{a-b})=l \]
   then by definition we have 
   \[ l\mid a-b  \quad and \quad  l\mid \frac{a^n-b^n}{a-b} \] since l divides both hand side of the equation then consequently it has to divides $nb^{n-1}.$
   Same holds for the prime factors in prime factorization of $l$, indeed:
   
   Let $l= p_1 ^ {\alpha_1} \dots p_m ^{\alpha_m}$ then for each i;
   
   \[ p_i\mid a-b \quad  and \quad p_i\mid \frac{a^n-b^n}{a-b}
  \text{ hence, } p_i\mid nb^{n-1}.   \] Clearly, since $( p_i, b ) = 1$, we get
   
   \[  p_i\mid n \quad that\quad implies \quad p_i\in \{ 1,n \} \]
   
   Now according to the assumption of the lemma ,i.e, $ n\nmid a^n-b^n$ , $l$ has to be 1, which accomplishes the proof.

\end{proof}

\[\]

\begin{theorem}[ Fermat's Last theorem of special case ]

Suppose a,b and c $\in \mathbb{N}$ and are relatively prime ,i.e, $(a,b,c)=1$  and for each prime n that $4\mid(n+1)$ then if $n\nmid abc$ the equation \[ a^{2n}=b^{2n}+c^{2n}  \] does not have any solution in natural numbers.

\end{theorem}

\begin{proof}

By lemma 1, there exist r and s $\in \mathbb{N}$ 

\[ a^n = r^2 + s^2 ,\quad and \quad c^n= r^2-s^2 \quad and \quad b^n=2rs   \qquad (1) \]

On the other hand, \[ a^{2n} -c^{2n}= b^{2n} \Longrightarrow (a^2)^n -(c^2)^n =b^{2n}  \] then by expanding above we get:

\[ (a^2-b^2) \biggl((a^2)^{n-1} +(a^2)^{n-2} (c^2) + (a^2)^{n-3} (c^2)^2 + \dots \]
\[+ (a^2)^2 (c^2)^{n-3} + (a^2)(c^2)^{n-2} + (c^2)^{n-1} \biggl) = b^{2n} \]
  
  By analyzing two above parenthesis and invoking lemma 2 and 3 we obtain :
  
  \[  \sqrt[2n]{a^2-c^2}= k   \] for some $k \in \mathbb{N}.$
  
  which implies \[ a^2=c^2 + (k^n)^2 \] by which and lemma 1 we again can find $l,t \in \mathbb{N} $ such that
  
  \[ a=l^2+ t^2 \quad, c=l^2-t^2, \quad k^n= 2lt \]
  
  Here is worth mentioning that $l$ and $t$  can not be at same time odd or even, because otherwise $(a,c)\neq 1.$
  
  it is not hard to derive from above equations and the equations  (1) , the followings:

  \[ a^n= (l^2+ t^2)^n = r^2 + s^2 \qquad  and \qquad c^n= (l^2-t^2 )^n= r^2 - s^2\]
  
  we now proceed the rest of the proof by expanding above equations:
  
  \[ (l^2)^n + \binom{n}{1}(l^2)^{n-1}(t^2) + \binom{n}{2}(l^2)^{n-2} (t^2)^2 + \dots +\binom{n}{n-2}(l^2)^2 (t^2)^{n-2} + \]
  \[ +\binom{n}{n-1}(l^2)(t^2)^{n-1} + (t^2)^n = r^2+s^2 \qquad (2)\]
  
  Similarly,
  
 \[ (l^2)^n - \binom{n}{1}(l^2)^{n-1}(t^2) + \binom{n}{2}(l^2)^{n-2} (t^2)^2 - \dots -\binom{n}{n-2}(l^2)^2 (t^2)^{n-2} + \]
  \[ \binom{n}{n-1}(l^2)(t^2)^{n-1} - (t^2)^n = r^2-s^2 \qquad (3)\]
  
  So now we add the equations (2) and (3) and recall an elementary fact that square of any odd number can be written as $8q+1$ for some $q\in \mathbb{N}$ we obtain:
  
  \[ (l^2)^{n-1} + \binom{n}{2}(l^2)^{n-3}(t^2)^2+ \dots + \binom{n}{n-3}(l^2)^2(t^2)^{n-3} \]
  \[ + \binom{n}{n-1}(t^2)^{n-1} = \frac{r^2}{l^2} = 8m+1 \qquad (4) \] for some $m\in \mathbb{N}.$
  
  It is noticeable that in above we factored out the term $l^2.$ Please note that whether $l$ or $t$ are odd or even or the other way round, the left hand side of the (4) is odd.
  
  We follow by the subtracting equations (2) and (3):
  
  \[ \binom{n}{1} (l^2)^{n-1} + \binom{n}{3}(l^2)^{n-3}(t^2)^2 + \dots + \binom{n}{n-2}(l^2)^2(t^2)^{n-3} +  \]
  \[ +(t^2)^{n-1} = \frac{s^2}{t^2}= 8m'  +1  \qquad (5) \] for some $m'\in \mathbb{N}.$
  
  Again we point that in above calculations the term $t^2$ has been factored. Please note that whether $l$ or $t$ is odd or even or the other way round, the left hand side of the (5) is odd.
  \[\]
  Then eventually we sum up the equations (4) and (5) to get to the following equation
  
  \[ (l^2)^{n-1} (n+1) + \left(\binom{n}{2} + \binom{n}{3}\right) (l^2)^{n-3}(t^2)^2 + \dots +  \]
  \[ \left( \binom{n}{n-3} + \binom{n}{n-2}\right) (l^2)^{2}(t^2)^{n-3} + (t^2)^{n-1} (n+1)  \] 
  \[ = 8(m+m')+2  \qquad (6) \]

  By assumption the first term is dividable by 4, moreover, whether $t$ or $l$ is even or odd or the other way round, there will be a factor 4 among all terms on the left hand side of the above equation, so then the LHS is dividable by 4. However, on the right hand side there is a constant 2 which causes the contradiction.

\end{proof}

\begin{remark}

The hypothesis $(a,b,c)=1$ is just redundant, because if it is not we cancel the common factor and then follow the proof as above.
\end{remark}

\cite{MR1509441}
\cite{Barbara2016-ls} 
\cite{MR990017} 
\cite{riehl2005kummer}
\cite{MR532370}
\printbibliography 

\end{document}